\documentclass{amsart}

\newtheorem{theorem}{Theorem}[section]

\newtheorem{corollary}[theorem]{Corollary}

\theoremstyle{definition}

\theoremstyle{remark}

\numberwithin{equation}{section}

\begin{document}

\title{Recurrence formula for some higher order evolution equations}


%    author one information
\author{Yoritaka Iwata}
\address{Osaka University of Economics and Law, Gakuonji, Yao, Osaka 581-0853, Japan}
\curraddr{}
\email{iwata\_phys@08.alumni.u-tokyo.ac.jp}
\thanks{}

%    author two information
%\author{}
%\address{}
%\curraddr{}
%\email{}
%\thanks{}

%    \subjclass is required.

\date{}

\dedicatory{}

%    Abstract is required.
\begin{abstract}
Riccati's differential equation is formulated as abstract equation in finite or infinite dimensional Banach spaces.
Since the Riccati's differential equation with the Cole-Hopf transform shows a relation between the first order evolution equations and the second order evolution equations, its generalization suggests the existence of recurrence formula leading to a sequence of differential equations with different order.
%%%
In conclusion, by means of the logarithmic representation of operators, a transform between the first order evolution equations and the higher order evolution equation is presented.
Several classes of evolution equations with different orders are given, and some of them are shown as examples.
\end{abstract}

\maketitle



\section{Introduction}
%%%%%%  
Let $u(t) \in C^0([0,T]; X)$ satisfy a first order evolution equation $\partial_t u(t) = A_1(t) u(t)$ in a Banach space $X$.
In this situation, the associated evolution operator $U(t,s)$ satisfying
\[
u(t) = U(t,s) u_s, \quad   u(s) = u_s \in X
\]
is assumed to exist as a two-parameter $C^0$-semigroup.
Let $k$ be a positive integer. 
For a given solution $u(t)$ in a fixed interval $t \in [0, T]$, much attention is paid to find high order evolution equations
\begin{equation} \label{kthmastereq}
\partial_t^k u(t) = A_k(t) u(t) 
~ \Leftrightarrow ~ 
\partial_t^k U(t,s) u_s = A_k(t) U(t,s) u_s  \qquad
\end{equation}
which is satisfied by exactly the same $u(t) = U(t,s) u_s$ of the first order evolution equation, and therefore by the evolution operator $U(t,s)$.
According to the preceding results \cite{17iwata-1,17iwata-3,20iwata-3},  a set of operators $\{ A_k(t) \}_{0 \le t \le T}$ is represented by the logarithm of operators (for historical milestones of logarithm of operators under the sectorial  assumption, see \cite{94boyadzhiev,03hasse,06hasse,69nollau,00okazawa-1,00okazawa-2}).
On the other hand, abstract representations of Cole-Hopf transform \cite{51cole,50hopf} and the Miura transform \cite{68miura} have been obtained by the logarithmic representation of operators  \cite{19iwata-2,20iwata-1}, and they are to be unified in one recurrence formula.
In the recurrence formula, the Cole-Hopf transform corresponds to the first order relation, and the Miura transform to the second order relation.

In this paper, by generalizing the logarithmic representation  (for its physical applications, see \cite{19iwata-1,18iwata-2, 22iwata-1}) providing nonlinear transforms, a relation between the solutions of first order evolution equations and second order evolution equations is generalized to a general $k$-th order relation.
It is equivalent to profile unknown $k$-th order operator $A_k(t)$ from the first order operator $A_1(t)$.
That is to say the presented recurrence formula brings about a kind of unification inside unbounded operator theory.

\section{Mathematical settings}
Let $X$ be a Banach space and $B(X)$ be a set of bounded operators on $X$.
The norms of both $X$ and $B(X)$ are denoted by $\| \cdot \|$ if there is no ambiguity.
Let $t$ and $s$ be real numbers included in a finite interval $[0, T]$, and $U(t,s)$ be the evolution operator in $X$. 
The two parameter semigroup $U(t,s)$, which is continuous with respect to both parameters $t$ and $s$, is assumed to be a bounded operator on $X$.
That is, the boundedness condition 
\begin{equation} \label{bound} \begin{array}{ll}
\| U(t,s) \| \le M e^{\omega (t-s)}
\end{array} \end{equation}
is assumed.
Following the standard theory of abstract evolution equations~\cite{61kato,70kato,73kato,60tanabe,61tanabe,79tanabe}, the semigroup property:
\[ \begin{array}{ll}
 U(t,r) U(r,s) = U (t,s)  
\end{array} \]
is assumed to be satisfied for arbitrary $s \le r \le t$ included in a finite interval $[0, T]$.
Let the evolution operator $U(t,s)$ be generated by $A_1(t)$.
Then, for certain functions $u(t) \in C^0([0,T]; X)$,
\begin{equation} \label{1stmastereq}
\partial_t u(t) = A_1(t) u(t)
\end{equation}
is satisfied in $X$.
That is, the operator $A_1(t)$ is an infinitesimal generator of $C^0$-semigroup $U(t,s)$.

\section{Back ground and basic concepts}
\subsection{Logarithmic representation of operators}
According to the preceding work \cite{17iwata-1,17iwata-3,20iwata-3} dealing with the logarithmic representation of operators, the evolution operator $U(t,s)$ is assumed to be generated by $A_1(t)$.
Let $\psi$ satisfy Eq.~\eqref{1stmastereq}.
The solution $\psi$ is generally represented by $\psi(t) = U(t,s) u_s$ for a certain $u_s \in X$, so that the operator equality is obtained by identifying $\psi(t)$ with $U(t,s)$, that is, identifying a function with an operator.
For instance, it is practical to imagine that $\psi(t)$ is represented by
\begin{equation}  \label{expop}
\psi(t)  =  \exp \left( \int_s^t A_1(\tau) d \tau  \right).
\end{equation}
Here the integral representation  \eqref{expop} is valid at least if $A_1(t)$ is $t$-independent, whose validity in $t$-independent cases is shown in Appendix I.
Let $\kappa$ be a certain complex number.
Using the Riesz-Dunford integral \cite{43dunford}, the infinitesimal generator $A_1 (t)$ of the first order evolution equation is written by
\begin{equation} \label{mastart}
A_1(t) =  \psi^{-1}   \partial_t  \psi
=   (I + \kappa U(s,t))
\partial_t  {\rm Log} ( U(t,s) + \kappa I)
\end{equation}
under the commutation, where ${\rm Log}$ means the principal branch of logarithm, and $U(t,s)$ is temporarily assumed to be a group (i.e., existence of $U(s,t) = U(t,s)^{-1}$ is temporarily assumed to be valid for any $0 \le s \le t \le T$), and not only a semigroup. 
The validity of \eqref{mastart} is confirmed formally by 
\[ \begin{array}{ll}
(I + \kappa U(s, t) ) \partial_t {\rm Log}(U(t, s) +  \kappa I)   \vspace{1.5mm}\\
= U(s, t)(U(t, s) + \kappa I)\partial_t U(t, s)(U(t, s) + \kappa I)^{-1}  \vspace{1.5mm}\\
= U(s, t) \partial_t U(t, s)   \vspace{1.5mm}\\
 = U(s, t)A_1(t)U(t, s) = A_1(t)
\end{array} \]
under the commutation assumption.
This relation is associated with the abstract form of the Cole-Hopf transform \cite{19iwata-2}.
Indeed the correspondence between $\partial_t  \log \psi$ and $A_1(t)$ can be understood by $U(s, t) \partial_t U(t, s)   = A_1(t)$ shown above.
Indeed,
\[
 \psi^{-1}   \partial_t  \psi   =   \partial_t  \log \psi 
 \quad  \Rightarrow \quad  
  \partial_t  \psi   =  ( \partial_t  \log \psi) \psi 
\]
is valid under the commutation assumption.

By introducing alternative infinitesimal generator $a(t,s)$ \cite{17iwata-3} satisfying
\[
e^{a(t,s)} := U(t,s)+ \kappa I, 
\]
a generalized version of the logarithmic representation
\begin{equation}  \label{altrep}
 A_1(t) =  \psi^{-1}   \partial_t  \psi
=   ( I   - \kappa   e^{-a(t,s)}  )^{-1}   \partial_t  a(t,s)
\end{equation}
is obtained, where $U(t,s)$ is assumed to be only a semigroup. 
The right hand side of Eq.~(\ref{altrep})  is actually a generalization of \eqref{mastart}; indeed, by only assuming $U(t,s)$ as a semigroup defined on $X$,  $e^{-a(t,s)}$ is always well defined by a convergent power series, and there is no need to have a temporal assumption for the existence of $U(s,t)= U(t,s)^{-1}$. 
It is remarkable for $e^{a(t,s)}$ being generated by $a(t,s)$ that $e^{-a(t,s)} = e^{a(s,t)}$ is not necessarily satisfied~\cite{17iwata-3}.
The validity of Eq.~\eqref{altrep} is briefly seen in the following.
Using the generalized representation, 
\[ \begin{array}{ll}
e^{a(t,s)} = U(t, s) + \kappa I \quad \Rightarrow \quad a(t, s) = {\rm Log}(U(t, s) + \kappa I) \vspace{1.5mm} \\
\qquad  \Rightarrow \quad \partial_t a(t, s) = \left[ \partial_t U(t, s) \right]  (U(t, s) + \kappa I)^{-1} 
= A_1(t)U(t, s)(U(t, s) + \kappa I)^{-1}
\end{array} \]
is obtained under the commutation assumption between $A_1(t)$ and $U(t,s)$.
It leads to
\[ \begin{array}{ll}
(I - \kappa e^{-a(t,s)})^{-1} \partial_t a(t,s)  =
(I - \kappa e^{-a(t,s)})^{-1} A_1(t)U(t, s)(U(t, s) + \kappa I)^{-1},
\end{array} \]
and therefore
\[ \begin{array}{ll}
A_1 (t) =   (U(t,s) + \kappa I )  U(t,s) ^{-1}  A_1(t)U(t, s)(U(t, s) + \kappa I)^{-1}   \vspace{1.5mm} \\
\quad \Leftrightarrow \quad
A_1 (t) = (I - \kappa e^{-a(t,s)})^{-1} A_1(t) (I - \kappa  e^{-a(t,s)} )
\end{array} \]
is valid under the commutation assumption.
It simply shows the consistency of representations using the alternative infinitesimal generator $a(t,s)$.
In the following, the logarithmic representation \eqref{altrep} is definitely used, and the original representation \eqref{mastart} appears if it is necessary. 
It is notable for bounded/unbounded operators holding the above representation that, based on ordinary and generalized logarithmic representations, the algebraic property of set of infinitesimal generator is known \cite{18iwata,20iwata-2}.

\subsection{Miura transform}
The Miura transform, which holds the same form as the Riccati's differential equation, is represented by
\[
u = \partial_x v + v^2,
\]
where $u$ is a solution of the modified Korteweg-de Vries equation (mKdV equation), and $v$ is a solution of the Korteweg-de Vries equation (KdV equation).
For the Riccati's differential equation, $u$ is a function standing for an inhomogeneous term, and $v$ is unknown function.
It provides a representation of the infinitesimal generator of second order differential equation.
Indeed, if the Miura transform is combined with the Cole-Hopf transform $v = \psi^{-1} \partial_x \psi$, it is written by
\begin{equation} \label{basic} \begin{array}{ll}
u  = \partial_x ( \psi^{-1} \partial_x \psi ) + ( \psi^{-1} \partial_x \psi)^2  \vspace{2.5mm} \\
= - \psi^{-2} (\partial_x \psi)^2+  \psi^{-1}  \partial_x ( \partial_x \psi ) +  \psi^{-2} (\partial_x \psi)^2  \vspace{2.5mm} \\
= \psi^{-1} \partial_x^2 \psi,
\end{array} \end{equation}
where the commutation between $\psi$,  $\partial_x \psi$ and $\partial_x^2 \psi$ is assumed.
This issue should be carefully treated in operator situation; i.e., in the standard theory of abstract evolution equations of hyperbolic type \cite{70kato,73kato}, $x$-dependent infinitesimal generators (corresponding to $\partial_x \psi$ or $\partial_x^2 \psi$ respectively) do  not generally commute with the  evolution operator  (corresponding to $\psi$ or $\partial_x \psi$ respectively), although such commutations are always true in $x$-independent infinitesimal generators.
For sufficiently smooth $\psi$ in the $x$-direction, one of the implication here is that the function $\psi$ satisfies both the second order equation
\[ \begin{array}{ll}
u  = \psi^{-1} \partial_x^2 \psi   
\quad \Leftrightarrow \quad
 \partial_x^2 \psi =  u  \psi   
\end{array} \]
and the first order equation
\[ \begin{array}{ll}
v = \psi^{-1} \partial_x \psi
\quad \Leftrightarrow \quad
 \partial_x \psi =  v  \psi   ,
\end{array} \]
at the same time under the commutation assumption.
The combined use of Miura transform and Cole-Hopf transform is called the combined Miura transform in \cite{20iwata-1} (for the combined use in the inverse scattering theory, e.g., see \cite{81ablowitz}).
Provided the solvable first order autonomous differential equation (i.e., the Cole-Hopf transform) with its solution $\psi$, the Miura transform shows a way to find the second order autonomous differential equation to be satisfied by exactly the same $\psi$. 

The combined Miura transform $ \partial_x^2 \psi = u \psi$ can be generalized as the second order abstract equation $ \partial_t^2 \psi = A_2(t) \psi $ in finite or infinite dimensional Banach spaces by taking $u$ as a closed operator $A_2(t):D(A_2) \to X$ in a Banach space $X$, where the index $2$ denotes the order of differential equations, and the notation of variable is chosen as $t$.
The solution $\psi$ is generally represented by $\psi(t) = U(t,s) u_s$ for a certain $u_s \in X$, so that the operator version of the combined Miura transform is formally obtained by assuming $\psi(t) = U(t,s)$.
\begin{equation} \label{intermed} \begin{array}{ll} 
A_1(t) = \psi^{-1}  \partial_t \psi  \vspace{1.5mm} \\
 \qquad  = \partial_t \log \psi    = \partial_t \log U(t,s)  \vspace{1.5mm} \\
 \qquad  =  ( I   - \kappa   e^{-a(t,s)}  )^{-1}   \partial_t  a(t,s)
\end{array}  \end{equation}
is valid under the commutation assumption between $\psi = U(t,s)$ and $\partial_t \psi = \partial_t U(t,s)$.
According to the spectral structure of $U(t,s)$, its logarithm ``$\log U(t,s)$" cannot necessarily be defined by the Riesz-Dunford integral.
However, the logarithm ``$\log e^{a(t,s)}$" is necessarily well defined  by the Riesz-Dunford integral, because $a(t,s)$ with a certain $\kappa$ is always bounded on $X$ regardless of the spectral structure of bounded operator $U(t,s)$.
That is, it is necessary to introduce a translation (i.e., a certain nonzero complex number $\kappa$) for defining logarithm functions of operator.
Here it is not necessary to calculate $\log  (U(t,s))$ at the intermediate stage, and only the most left hand side and the most right hand side of Eq.~\eqref{intermed} make sense.

In terms of applying nonlinear transforms such as the Miura transform and the Cole-Hopf transform, it is necessary to identify the functions with the operators, which is mathematically equivalent to identify elements in $X$ with elements in $B(X)$.
In other words, it is also equivalent to regard a set of evolution operators as a set of infinitesimal generators.
%%%
The operator representation of the infinitesimal generator $A_2(t)$ of second order evolution equations has been obtained in \cite{20iwata-1}.
Under the commutation assumption between  $\psi$ and $\partial_t \psi$ and that between $\partial_t \psi$ and $\partial_t^2 \psi$, the representation of infinitesimal generator $A_2(t)$ in Eq.~(\ref{cm2nd}) is formally obtained by
\begin{equation}  \label{cm3rd}  \begin{array}{ll} 
A_2(t)   = \psi^{-1} \partial_t^2 \psi
 =  \left[ \psi^{-1} \partial_t \psi \right]  ~  \left[ (\partial_t \psi) ^{-1}  \partial_t^2 \psi \right] 
 = \left[ \partial_t \log \psi \right] ~ \left[ \partial_t \log (\partial_t \psi) \right]    \vspace{1.5mm} \\
\qquad  =   \left[ \partial_t \log U(t,s) \right]  ~ \left[ \partial_t \log (\partial_t U(t,s)) \right]  \vspace{1.5mm} \\
\qquad  =   ( I   - \kappa   e^{-a(t,s)}  )^{-1}   \partial_t  a_1(t,s)   ~   ( I   - \kappa   e^{{u a}(s,t)}  )^{-1}   \partial_t  a_2 (t,s)  
\end{array} \end{equation}
in $X$, where alternative infinitesimal generators $a_1 (t,s)$ and $a_2 (t,s)$, which are defined by 
\[ e^{a_1 (t,s)} = U(t,s) + \kappa I ,  
\quad
e^{a_2 (t,s)} = \partial_t U(t,s) + \kappa I,    \] 
and therefore by
\[ a_1 (t,s) =  {\rm Log} (  U(t,s) + \kappa I ),
\quad
a_2 (t,s) =  {\rm Log} ( \partial_t U(t,s) + \kappa I ),   \] 
generate  $e^{ a_1 (t,s)}$ and $e^{ a_2 (t,s)}$, respectively.
For $\psi(t) = U(t,s) u_s \in C^0([0,T];X)$, let generally unbounded operator $\partial_t U(t,s)$ be further assumed to be continuous with respect to $t$; for the definition of logarithm of unbounded evolution operators by means of the doubly-implemented resolvent approximation, see \cite{20iwata-3}. 
Note that the continuous and unbounded setting for $\partial_t U(t,s)$ is reasonable with respect to $C^0$-semigroup theory, since both $\psi (t) = U(t,s) u_s$ and $\partial_t U(t,s) u_s$ are the two main components of solution orbit defined in the infinite-dimensional dynamical systems.
Similar to Eq.~\eqref{intermed}, it is necessary to introduce a translation to define $\log U(t,s)$.
Consequently the most right hand side of Eq.~\eqref{cm3rd} is a mathematically valid representation for $A_2(t)$, where it is not necessary to calculate $\log U(t,s)$ and $\log \partial_t U(t,s)$ at the intermediate stage.
In this manner, the infinitesimal generators of second order evolution equations are factorized as the product of two logarithmic representations of operators: $\partial_t \log U(t,s)$ and $\partial_t \log \partial_t U(t,s)$ (for the details arising from the operator treatment, see \cite{20iwata-1}).
The order of $\partial_t \log U(t,s)$ and $\partial_t \log \partial_t U(t,s)$ can be changed independent of the boundedness/unboundedness of operator, as it is confirmed by the commutation assumption
\[
\psi^{-1} \partial_t^2 \psi = ( \partial_t^2 \psi) \psi^{-1}    = ( \partial_t^2 \psi) (\partial_t \psi)^{-1}  (\partial_t \psi) \psi^{-1} = [ (\partial_t \psi)^{-1}   \partial_t^2 \psi ] [ \psi^{-1} \partial_t \psi ]. \] 
This provides the operator representation of the combined Miura transform.
In unbounded operator situations, the domain space of infinitesimal generators should be carefully discussed. 
Indeed, the domain space of $A_2(t)$, which must be a dense subspace of $X$, is expected to satisfy
\begin{equation} \begin{array}{ll}  \label{cm1} 
D(A_2(t)) = \left\{  u \in X; ~ \{  \partial_t {\hat a}(t,s)  u \} \subset D( \partial_t a(t,s) )  \right\}
\subset X,   \vspace{1.5mm} \\
D( \partial_t a(t,s) ) =  \left\{  u \in X; ~ \{ \partial_t a(t,s) u \} \subset  X  \right\}
\subset X
\end{array} \end{equation}
or
\begin{equation}  \begin{array}{ll}  \label{cm2} 
D(A_2(t)) = \left\{  u \in X; ~ \{  \partial_t a (t,s)  u \} \subset D( \partial_t {\hat a}(t,s) )  \right\}
\subset X,   \vspace{1.5mm} \\
D( \partial_t {\hat a}(t,s) ) =  \left\{  u \in X; ~ \{ \partial_t {\hat a}(t,s) u \} \subset  X  \right\}
\subset X
\end{array} \end{equation}
depending on the order of product, where note that both $a(t,s)$ and $\partial_t a(t,s)$ depend on $t,s \in [0,T]$.

Consequently, for solutions $\psi(t) = U(t,s) u_s$ satisfying $ \partial_t \psi = A_1(t) \psi $, 
the second order evolution equation
\begin{equation} \label{cm2nd}  \begin{array}{ll}
 \partial_t^2 \psi = A_2(t) \psi 
\end{array} \end{equation}
is also satisfied by setting the infinitesimal generator $A_2(t)$ as defined by the combined Miura transform \eqref{basic}.
It means that $A_2(t)$ is automatically determined by a given operator $A_1(t)$. 
In the following, beginning with the second order formalism (i.e., the combined Miura transform),  the relation is generalized as the recurrence formula for defining the higher order operator $A_k$ ($k \ge 3$).

\section{Main result}

\subsection{Recurrence formula generalizing the combined Miura transform} 
For exactly the same $\psi(t) = U(t,s)$ satisfying $ \partial_t \psi = A_1(t) \psi$, a generally-unbounded operator $\partial_t^k A$ be continuous with respect to $t$ ($k \ge 1$: integer).
Let the infinitesimal generator of $k$-th order evolution equations be defined by
\[
A_k =   \psi^{-1}  \partial_t^{k}  \psi
\]
under the commutation assumption.
Among limited numbers of preceding works, a lecture note summarizing the theory of higher order abstract equations, see \cite{98xiao}.
For defining the higher-order infinitesimal generator in an abstract manner, a recurrence formula is introduced.
In this section, the infinitesimal generators $A_k(t)$ are assumed to be general ones holding the time dependence.

\begin{theorem} \label{thm01}
For $t \in [0,T]$, let generally-unbounded closed operators $A_1(t), A_2(t), \cdots A_n(t)$ be continuous with respect to $t$ defined in a Banach space $X$. 
Here $A_1(t)$ is further assumed to be an infinitesimal generator of the first-order evolution equation
\begin{equation} \label{cp1} \begin{array}{ll}
   \partial_t \psi(t) = A_{1}(t) \psi(t), 
\end{array} \end{equation}
in $X$, where $\psi(t)$ satisfying the initial condition $\psi(0) = \psi_0 \in X$ is the solution of the Cauchy problem of \eqref{cp1}.
The commutation between $\psi$ and $\partial_t \psi$, and that between $\psi$ and $\partial_t^n \psi$ are assumed to be valid ($n \ge 2$).
Then the infinitesimal generator of the $n$-th order evolution equation
\begin{equation} \label{cp2} \begin{array}{ll}
   \partial_t^n \psi(t) = A_{n} (t) \psi(t)
\end{array} \end{equation}
is given by the recurrence formula
\begin{equation} \label{recc} \begin{array}{ll}
A_{n}(t) = (\partial_t + A_{1}(t) ) A_{n-1}(t)  \vspace{1.5mm} \\
\end{array} \end{equation}
being valid for $n \ge 2$, where $\psi(t)$ is common to Eqs.~\eqref{cp1} and \eqref{cp2}.
\end{theorem}

\begin{proof}
The statement is proved by the mathematical induction.
Let $\psi$ satisfy $\partial_t \psi(t) = A_{1}(t) \psi(t) $.
In case of $n=2$, let $A_2$ satisfy $\partial_t \psi(t) = A_{2}(t) \psi(t)$.
The infinitesimal generator $A_2(t)$ is formally represented by
\[
A_{2} (t) = \partial_t A_{1}(t) + A_{1}^2(t),
\]
as it is readily understood by the combined Miura transform (see also \eqref{basic}).
Let the relation for $n=k-1$: 
\[
A_{k-1}(t) = \psi(t)^{-1} \partial_{t}^{k-1} \psi(t)
\]
be satisfied.
By substituting $A_{k-1} (t)= \psi^{-1} \partial_{t}^{k-1} \psi$ and $ A_{1}(t) = \psi^{-1} \partial_{t} \psi$ into
\[
A_k(t) = \partial_t A_{k-1}(t) + A_1(t) A_{k-1}(t),
\]
it results in
\[ \begin{array}{ll}
A_k(t) = \partial_t (\psi^{-1} \partial_{t}^{k-1} \psi) + (\psi^{-1} \partial_{t} \psi) ( \psi^{-1} \partial_{t}^{k-1} \psi)  \vspace{1.5mm} \\
\quad =-  \psi^{-2} (\partial_t \psi) (\partial_{t}^{k-1} \psi)  + \psi^{-1}  \partial_{t}^{k} \psi
+ (\psi^{-1} \partial_{t} \psi) ( \psi^{-1} \partial_{t}^{k-1} \psi) \vspace{1.5mm} \\
\quad =  \psi^{-1}  \partial_{t}^{k} \psi
\end{array}  \]
under the commutation assumption.
Consequently,
\[   \partial_{t}^{k} \psi(t) = A_k(t) \psi(t) \]
is obtained.
\end{proof}

Using Eq.~(\ref{recc}), the operator version of Riccati's differential equation is obtained if $n=2$ is applied.
The resulting equation
\[ \begin{array}{ll}
A_{2}(t) = \partial_t A_{1}(t) + [A_{1}(t)]^2  \vspace{1.5mm} \\
\end{array} \]
is the Riccati type differential equation to be valid in the sense of operator.
The domain space of $A_2(t)$ is determined by $A_1(t)$, and not necessarily equal to the domain space of $A_1(t)$.  
Consequently, the Riccati type nonlinear differential equation is generalized to the operator equation in finite/infinite-dimensional abstract Banach spaces. 
The obtained equation 
\[ A_{n}(t) = (\partial_t + A_{1}(t) ) A_{n-1}(t)  \]
itself is nonlinear if $n=2$, while linear if $n\ne 2$. 
Its dense domain, which is also recursively determined, is assumed to satisfy
\[  \begin{array}{ll} 
D(A_n(t)) = \left\{  u \in X; ~ \{  A_{n-1}  u \} \subset D( A_1(t) )  \right\}
\subset X,  
\end{array} \]
for $n \ge 2$, with
\[  \begin{array}{ll} 
D( A_1(t) ) =  \left\{  u \in X; ~ \{ A_1 u \} \subset  X  \right\} \subset X.
\end{array} \]
In the following examples, concrete higher order evolution equations are shown.
In each case, the recurrence formula plays a role of transform between the first order equation and $n$th order equation. 
\vspace{2.5mm} \\
%%%
{\bf Example 1. [2nd order evolution equation].} ~The second order evolution equations, which is satisfied by a solution $\psi$ of first order evolution equation $\partial_t \psi = \partial_x^k \psi$, is obtained.
The operator $A_1$ is given by $t$-independent operator $\partial_x^k$ ($k$ is a positive integer).
\[ \begin{array}{ll}
A_{2} = (\partial_t + \partial_x^k ) \partial_x^k  
= \partial_t  \partial_x^k + \partial_x^{2k},  \vspace{1.5mm} \\
\end{array} \]
and then, by applying Eq.~(\ref{recc}), the second order evolution equation
\begin{equation}  \label{evo02} \begin{array}{ll}
\partial_t^2 u = \partial_t  \partial_x^k u + \partial_x^{2k} u 
\end{array} \end{equation}
is obtained.
Indeed, let $u$ be a general or special solution of $\partial_t u = \partial_x^k u$, the validity of statement
\[ \begin{array}{ll}
\partial_t u = \partial_x^k u
\qquad \Rightarrow \qquad
\partial_t^2 u = \partial_t  \partial_x^k u + \partial_x^{2k} u 
\end{array} \]
is confirmed by substituting the operator equality $\partial_x^k  = u^{-1} \partial_t u$ and the associated equality
\[ \begin{array}{ll}
\partial_t  \partial_x^k  = \partial_t (u^{-1} \partial_t u)  = -(u^{-1} \partial_t u)^2 + (\partial_t^{2} u)  u^{-1},  \vspace{2.5mm} \\
 \partial_x^{2k}  = (u^{-1} \partial_t u)^2 \\
\end{array} \]
to the right hand side of Eq.~\eqref{evo02}, where the commutation between $u$, $\partial_t u$ and $\partial_t^2 u$ is assumed to be vald. 
Equation~\eqref{evo02} is a kind of wave equation in case of $k=1$.
Note that the obtained equation is a linear equation.
 \vspace{1.5mm} \\
%%%
{\bf Example 2. [3rd order evolution equation]} ~The third order evolution equations, which is satisfied by a solution $\psi$ of first order evolution equation $\partial_t \psi = \partial_x^k \psi$, is obtained in the same manner.
The operator $A_1$ is given by time-independent operator $\partial_x^k$  ($k$ is a positive integer).
\[ \begin{array}{ll}
A_{3} = (\partial_t + \partial_x^k ) ( \partial_t  \partial_x^{k} + \partial_x^{2k}) 
=  \partial_t  (\partial_t  \partial_x^k + \partial_x^{2k}) + 
\partial_x^k ( \partial_t  \partial_x^k + \partial_x^{2k}) ,  \vspace{1.5mm} \\
=   \partial_t^2  \partial_x^{k} + 2 \partial_t \partial_x^{2k}  
  + \partial_x^{3k} ,  \vspace{1.5mm} \\
\end{array} \]
and the third order evolution equation
\[ \begin{array}{ll}
\partial_t^3 u 
=  \partial_t^2  \partial_x^k u + 2 \partial_t \partial_x^{2k}  u  + \partial_x^{3k} u
\end{array} \]
is obtained.
The validity is confirmed by substituting the operator equality $\partial_x^k  = u^{-1} \partial_t u $.

\subsection{Logarithmic representation of $n$-th order infinitesimal generator}
Although the relation between $A_n$ and a given $A_1$ can be understood by Theorem 1, those representations and the resulting representations of evolution operator ($C_0$-semigroup) are not understood at this point.
In this section, utilizing the logarithmic representation of the infinitesimal generator, the representation of infinitesimal generator for the high order evolution equation is obtained.
Since the logarithmic representation has been known to be associated essentially with the first- and second-order evolution equations, the discussion in the present section clarifies a universal role of logarithm of operators, independent of the order of evolution equations.
 
\begin{theorem} \label{thm02}
For $t \in [0,T]$, let generally-unbounded closed operators $A_1(t), A_2(t), \cdots A_n(t)$ be continuous with respect to $t$ defined in a Banach space $X$. 
Here $A_1$ is further assumed to be an infinitesimal generator of the first-order evolution equation \eqref{cp1}.
For the $n$-th order evolution equations
\begin{equation}   \label{nthmastereq}
\partial_t^n u(t) = A_n(t) u(t)
\end{equation}
in $X$, the commutation between $\psi$, $\partial_t \psi$,  $\cdots \partial_t^n \psi$ are assumed to be valid.
Then $n$-th order operator $A_n(t)$ is represented by the product of logarithmic representations
\begin{equation}  \label{repu} \begin{array}{ll}
A_n (t) 
=   \Pi_{k=1}^n   \left[    (\kappa  {\mathcal U}_k(s,t) + I)
\partial_t  {\rm Log} (  {\mathcal U_k}(t,s) + \kappa I)   \right],
\end{array} \end{equation}
where $\kappa$ is a certain complex number, and  ${\mathcal U}_k(t,s)$ is the evolution operator of $ \partial_t^{k}  \psi = {\mathcal A}_k(t)   \partial_t^{k-1}  \psi $.
Note that the commutation between operators is assumed.
 In the operator situation, the commutation assumption is equivalent to assume a suitable domain space setting for each $ (\kappa  {\mathcal U}_k(s,t) + I) \partial_t  {\rm Log} (  {\mathcal U_k}(t,s) + \kappa I) $.
\end{theorem}

\begin{proof}
According to Theorem~\ref{thm01} and therefore to the operator version of the combined Miura transform, the $n$-th order  infinitesimal generator is regarded as $A_{n}(t) = \psi^{-1} \partial_t^{n} \psi$ under the commutation assumption.
In the first step, the $n$-th order infinitesimal generator is factorized as
\begin{equation} \label{recc2} \begin{array}{ll}
 \psi^{-1}   \partial_t^{n} \psi  \vspace{1.5mm} \\
=  \psi^{-1}  ( \partial_t^{n-1} \psi )     (  \partial_t^{n-1} \psi )^{-1}   \partial_t^{n} \psi   \vspace{1.5mm} \\
=  \psi^{-1}  ( \partial_t \psi )   ( \partial_t \psi )^{-1}     \cdots  ( \partial_t^{n-2} \psi )   ( \partial_t^{n-2} \psi )^{-1}    ( \partial_t^{n-1} \psi )     (  \partial_t^{n-1} \psi )^{-1}   \partial_t^{n} \psi  \vspace{1.5mm} \\
= \left[ \psi^{-1}  ( \partial_t \psi ) \right]  \left[ ( \partial_t \psi )^{-1}   ( \partial_t^{2} \psi )  \right]   \cdots  \left[ ( \partial_t^{n-2} \psi )^{-1}    ( \partial_t^{n-1} \psi ) \right]    \left[    (  \partial_t^{n-1} \psi )^{-1}   \partial_t^{n} \psi \right]    
\end{array} \end{equation}
where the commutation assumption is necessary for obtaining each $ \psi^{-1}   \partial_t^{k} \psi$ with $1 \le k \le n$.
Under the commutation assumption, the representation
\begin{equation} \label{recx3} \begin{array}{ll}
 A_n(t) = \psi^{-1}   \partial_t^{n} \psi  
 = \Pi_{k=1}^n   \left[    (  \partial_t^{k-1} \psi )^{-1}   \partial_t^{k} \psi \right] .
\end{array} \end{equation}
is obtained.

In the second step, the logarithmic representation is applied to $ (  \partial_t^{k-1} \psi )^{-1}   \partial_t^{k} \psi$.
The logarithmic representation for the first order abstract equation $ \partial_t {\tilde \psi} =   {\mathcal A}_k (t) {\tilde \psi}  \Leftrightarrow  \partial_t^{k} \psi =   {\mathcal A}_k(t) \partial_t^{k-1} \psi $ with ${\tilde \psi} = \partial_t^{k-1} \psi$ is
\begin{eqnarray*}
   {\mathcal A}_k (t) 
&=&  (  \partial_t^{k-1} \psi )^{-1}   \partial_t^{k}  \psi     \vspace{1.5mm} \\
&=&  (  \partial_t^{k-1}  {\mathcal U}_k(t,s) )^{-1}   \partial_t^{k}   {\mathcal U}_k(t,s)    \vspace{1.5mm} \\
&=&   \left[ (\kappa  {\mathcal U}_k(s,t) + I)
\partial_t  {\rm Log} (  {\mathcal U_k}(t,s) + \kappa I) \right] ,
\end{eqnarray*} 
where $ {\mathcal U}_k(t,s)$ is the evolution operator generated by ${\mathcal A}_k(t)$, and $  \psi =  {\mathcal U}_k(t,s) $ is applied for obtaining the operator equality. 
Again, the commutation assumption is necessary to obtain the logarithmic representation.
Consequently, the higher order infinitesimal generator becomes
\begin{eqnarray*} 
   {\mathcal A}_n (t)   &=& \Pi_{k=1}^n   \left[    (  \partial_t^{k-1}   {\mathcal U}_k(t,s) )^{-1}   \partial_t^{k}   {\mathcal U}_k(t,s)  \right]  \vspace{1.5mm} \\
 &=& \Pi_{k=1}^n   \left[   (\kappa  {\mathcal U}_k(s,t) + I)
\partial_t  {\rm Log} (  {\mathcal U_k}(t,s) + \kappa I)   \right]. 
\end{eqnarray*}
In this formalism, several possible orderings arise from the commutation assumption.
Since each component $\left[   (\kappa  {\mathcal U}_k(s,t) + I) \partial_t  {\rm Log} (  {\mathcal U_k}(t,s) + \kappa I)   \right]$ possibly unbounded in $X$, the choice of the domain space should be carefully chosen as discussed around Eqs.~\eqref{cm1} and \eqref{cm2}.
\end{proof}

Although the commutation assumption in Theorems \ref{thm01} and \ref{thm02} is restrictive to the possible applications, all the $t$-independent infinitesimal generators satisfy this property.
That is, the present results are applicable to linear/nonlinear heat equations, linear/nonlinear wave equations, and linear/nonlinear Schr\"odinger equations.
Consequently, a new path is introduced to the  higher order evolution equation in which the concept of "higher order" is reduced to the concept of "operator product" in the theory of abstract evolution equations.
%%%
Using the alternative infinitesimal generator being defined by $e^{\alpha_k(t,s)} = {\mathcal U}_k(t,s) + \kappa I$,
${\mathcal U}_k(t,s)$ can be replaced with $e^{\alpha_k(t,s)}$, and the following corollary is valid.

\begin{corollary} \label{cor03}
For $t \in [0,T]$, let generally-unbounded closed operators $A_1(t), A_2(t), \cdots A_n(t)$ be continuous with respect to $t$ defined in a Banach space $X$. 
Here $A_1$ is further assumed to be an infinitesimal generator of the first-order evolution equation \eqref{cp1}.
For the $n$-th order evolution equations
\[
\partial_t^n u(t) = A_n(t) u(t)
\]
in $X$, the commutation between $\psi$, $\partial_t \psi$,  $\cdots \partial_t^n \psi$ are assumed to be valid.
Then the $n$-th order operator is represented by the product of logarithmic representations
\begin{equation}  \label{repalt} \begin{array}{ll}
A_n (t) 
=   \Pi_{k=1}^n   \left[     ( I - \kappa   e^{-\alpha_k(t,s)} )^{-1} 
\partial_t  \alpha_k(t,s)    \right] ,
\end{array} \end{equation}
where $\kappa$ is a certain complex number, and  ${\mathcal U}_k(t,s)$ is the evolution operator generated by $  \partial_t^{k}  \psi = {\mathcal A}_k \partial_t^{k-1} \psi $ and $\alpha_k(t,s)$ is an alternative infinitesimal generator to $ {\mathcal A}_k$ satisfying the relation $e^{\alpha_k(t,s)} = {\mathcal U}_k(t,s) + \kappa I$.
Note that the commutation between operators is assumed, so that the order of logarithmic representation can be changed by assuming a suitable domain space settings.
\end{corollary}

\begin{proof}
The statement follows from applying
\[ e^{\alpha_k(t,s)} = {\mathcal U}_k(t,s) + \kappa I \]
to the representation shown in Theorem \ref{thm02}.
\end{proof}

Equation~\eqref{repalt} is actually a generalization of Eq.~\eqref{repu} as discussed around Eq.~\eqref{altrep}.

\subsection{$n$-th order generalization of Hille-Yosida type exponential function of operator} 
Let us take $t$-independent $n$-th order infinitesimal generator $A_n$.
The operators $A_k$ with $k=1,2, \cdots n$ are assumed to be the infinitesimal generator of $k$-th order evolution equations.
The specific equation for Eq.~\eqref{nthmastereq} is written as
\[
\omega^n - A_n = 0    
\]
by substituting a formal solution $e^{t \omega}$.
If the fractional power $A_n^{1/n}  $ of operator exists (for the definition of fractional powers of operators, see \cite{01carracedo}), 
$A_n^{1/n} $ is a root of the specific equation.
Furthermore, $A_n^{1/n} $ is assumed to be an infinitesimal generator of first order evolution equation $\partial_t u(t) = A_n^{1/n} u(t)$.
In this case, the specific equation is also written by
\[
\omega^n - A_n = 0  
\quad  \Leftrightarrow \quad
\left( \omega/A_n^{1/n} \right)^n -  I = 0.
\]
Note that the latter equation, which is called the cyclotomic equation in algebra \cite{1886scott,1836wantzel}, is known to hold the algebraic representation.
Consequently, based on the discussion made in the Appendix I, the integral representation of evolution operator is valid and one specific solution (more precisely, one of the fundamental solutions) is represented by
\begin{equation} \label{hy-gen}
{\mathcal U}_n(t,s) = \exp \left(t A_n^{1/n} \right)
= \exp \left( t ~  \left\{ \Pi_{k=1}^n   \left[     ( I - \kappa   e^{-a_k(t,s)} )^{-1} 
\partial_t  a_k(t,s)    \right]  \right\}^{1/n}  \right), 
\end{equation}
where $\kappa$ is a certain complex number, and  ${\mathcal U}_k(t,s)$ is the evolution operator generated by $  \partial_t^{k}  \psi = {\mathcal A}_k \partial_t^{k-1} \psi $ and $\alpha_k(t,s)$ is an alternative infinitesimal generator to $ {\mathcal A}_k$ satisfying the relation $e^{\alpha_k(t,s)} = {\mathcal U}_k(t,s) + \kappa I$.
The $n$-th order logarithmic representation is actually a generalization of the Hille-Yosida type generation theorem ($n=1$).
The representation shown in the most right hand side of \eqref{hy-gen} is always valid for certain $\kappa$, even if the fractional power $A_n^{1/n}  $ of operator is not well defined.
Here is an advantage of using alternative infinitesimal generators and the resulting logarithmic representation of unbounded operators.

\section{summary}
The recurrence formula for some higher order evolution equations is presented.
It connects the first order evolution equation with the higher order evolution equations.
By means of the logarithmic representation of the operators, the rigorous representation for the infinitesimal generators of evolution equations are obtained.
That is, 
\begin{itemize}
\item introduction of recurrence formula for obtaining a class of higher order equations: e.g.,  
\[  \begin{array}{ll}
\partial_t^2 u = \partial_t  \partial_x u + \partial_x^{2} u 
\end{array} \]
associated with $\partial_t u = \partial_x u$ (Eq.~\eqref{evo02} with $k=1$), where the introduced transform is represented by the recurrence formula generalizing the Miura transform (Theorem 1); \vspace{1.5mm}
\item higher order generalization of logarithmic representation of operators in which the concept of ``order of differential operator with respect to $t$" is reduced to the concept of ``multiplicity of operator product of infinitesimal generators". (Theorem 2);   \vspace{1.5mm}
\item generalization of Hille-Yosida type exponential function of operator (Eq.~\eqref{hy-gen})
\end{itemize}
have been done in this paper.
The present discussion shows another aspect of the Miura transform, which originally transform the solution of ``first-order" KdV equations to the solution of ``first-order" modified KdV equations.

\bibliographystyle{amsplain}

\begin{thebibliography}{30}
 \bibitem{81ablowitz}
 M.  J. Ablowitz  and  H.  Segur,
Solitons  and  the inverse  scattering  transform,
Philadelphia, SIAM, 1981.

 \bibitem{94boyadzhiev}
K.N. Boyadzhiev, 
Logarithms and imaginary powers of operators on Hilbert spaces, Collect.
Math. 45 (1994) no. 3, 287-300.

 \bibitem{01carracedo}
M. Carracedo and M. Sanz Alix, 
The theory of fractional powers of operators, 
North-Holland Mathematics Studies, vol. 187, North-Holland Publishing Co., Amsterdam, 2.

\bibitem{51cole}
J. D. Cole, 
On a quasi-linear parabolic equation occurring in aerodynamics, 
Quart. Appl. Math. 9 (1951) no. 3, 225-236.

\bibitem{43dunford}
N. Dunford, 
Spectral theory. I. Convergence to projections, 
Trans. Amer. Math. Soc. 54 (1943) 185-217.001.

 \bibitem{03hasse}
M. Haase, 
Spectral properties of operator logarithms, 
Math. Z. 245 (2003) no. 4, 761-779.

 \bibitem{06hasse}
M. Haase, 
The functional calculus for sectorial operators, 
Operator Theory: Advances and Applications, vol. 169, Birkhauser Verlag, Basel, 2006.

\bibitem{50hopf}
E. Hopf,
The partial differential equation $u_{t} + uu_{x} = \mu u_{xx}$, 
Comm. Pure Appl. Math. 3 (1950) 201-230.

\bibitem{17iwata-1}
Y. Iwata,
Infinitesimal generators of invertible evolution families,
Methods Funct. Anal. Topology {\bf 23} 1 (2017) 26-36. 

\bibitem{17iwata-3}
Y. Iwata,
Alternative infinitesimal generator of invertible evolution families, 
J. Appl. Math. Phys. {\bf 5} (2017) 822-830. 

\bibitem{18iwata}
Y. Iwata,
Operator algebra as an application of logarithmic representation of infinitesimal generators
J. Phys: Conf. Ser. 965 (2018) 012022. 

  \bibitem{18iwata-2} 
Y. Iwata, 
Unbounded formulation of the rotation group,
J. Phys: Conf. Ser. 1194 (2019) 012053. 

\bibitem{19iwata-2}
Y. Iwata,
Abstract formulation of the Cole-Hopf transform,
Methods Funct. Anal. Topology {\bf 25} 2 (2019) 142-151. 

\bibitem{19iwata-1}
Y. Iwata,
Relativistic formulation of abstract evolution equations,
AIP Conference Proceedings 2075  (2019) 100007. 

\bibitem{20iwata-3}
Y. Iwata,
Operator topology for logarithmic infinitesimal generators.
A chapter of a book "Structural topology and symplectic geometry", IntechOpen, 2020.

\bibitem{20iwata-1}
Y. Iwata,
Abstract formulation of the Miura transform, 
Mathematics 8 (2020) 747.

\bibitem{20iwata-2}
Y. Iwata,
Theory of $B(X)$-module: algebraic module structure of generally-unbounded infinitesimal generators,
Adv. Math. Phys. Vol. 2020, Article ID 3989572 (2020). 

  \bibitem{20iwata-3} 
Y. Iwata, 
Unbounded generalization of logarithmic representation of infinitesimal generators,
Math. Meth. Appl. Sci. 9002, 2023.

\bibitem{22iwata-1} 
Y. Iwata, 
Unbounded generalization of the Baker-Campbell-Hausdorff formulae, 
arXiv:2203.00378

\bibitem{61kato}
T. Kato,
Abstract evolution equations of parabolic type in Banach and Hilbort spaces, 
Nagoya Math. J. 19 (1961) 93-125.

\bibitem{70kato}
T. Kato, 
Linear evolution equations of hyperbolic-type,
J. Fac. Sci. Univ. Tokyo Sect. I 17 (1970) 241-258.

\bibitem{73kato}
T. Kato, 
Linear evolution equations of hyperbolic type. II,
J. Math. Soc. Japan 25 (1973) 648-666.

\bibitem{68miura}
R. Miura,
Korteweg-de Vries Equation and Generalizations I. a remarkable explicit nonlinear transformation, 
J. Math. Phys. 9 (1968) 1202-1204.

 \bibitem{69nollau}
V. Nollau, 
Uber den Logarithmus abgeschlossener Operatoren in Banachschen Raumen, 
Acta Sci. Math. (Szeged) 30 (1969) 161-174.

 \bibitem{00okazawa-1}
N. Okazawa, 
Logarithmic characterization of bounded imaginary powers, Semigroups of operators:
theory and applications (Newport Beach, CA, 1998), 
Progr. Nonlinear Differential Equations Appl., vol. 42, Birkhauser, Basel, 2000, pp. 229-237.

 \bibitem{00okazawa-2}
N. Okazawa, 
Logarithms and imaginary powers of closed linear operators, 
Integral Equations Operator Theory 38 (2000), no. 4, 458-500.

 \bibitem{1886scott}
A. C. Scott, 
The Binomial Equation $x^p-1=0$,
Amer. J. Math. 8, 261-264, 1886.

 \bibitem{60tanabe}
H. Tanabe, 
On the equations of evolution in a Banach space, 
Osaka Math. J., 12 (1960) 363-376.

 \bibitem{61tanabe}
H. Tanabe, 
Evolution equations of parabolic type, Proc. Japan Acad., 37, 10 (1961) 610-613.

 \bibitem{79tanabe}
H. Tanabe, 
Equations of evolution. Pitman, 1979. 

 \bibitem{1836wantzel}
M. L. Wantzel, 
Recherches sur les moyens de reconnaitre si un Probleme de Geometrie peut se resoudre avec la regle et le compas,
J. Math. Pures Appliq. 1, 366-372, 1836. 

\bibitem{98xiao}
T.-J. Xiao , J. Liang, 
The Cauchy Problem for Higher Order Abstract Differential
Equations, Lecture Notes in Math. vol. 1701, Springer, Berlin, 1998.


\end{thebibliography}


\newpage
\section*{Appendix I: Integral representation of evolution operator}
Here a short notice is made for the integral representation of $U(t,s)$, where the existence of evolution operator $U(t,s)$ is assumed in the present paper.
Generation of evolution operator or $C_0$-semigroup is definitely understood by the Hille-Yosida theorem in which the evolution operator is represented by the exponential function of operators.
Let $t,s$ be included in a certain interval $[0,T]$.
In particular, it is readily seen by $t$-independent case when $A_1(t) \equiv A_1$.
\[
\lambda I  - \int_s^t A_1(\tau) d \tau
= \lambda I - (t-s) A_1 
= (t-s) (\lambda (t-s)^{-1} I  - A_1),
\]
so that the estimate of resolvent operator
\[ \begin{array}{ll}
\left\| \left( \lambda I - \int_s^t A_1(\tau) d \tau  \right)^{-1}  \right\|
= (t-s)^{-1}   \left\|  ( \lambda(t-s)^{-1} I - A_1 )^{-1} \right\|     \vspace{1.5mm} \\
\le  \displaystyle {  (t-s)^{-1}   \frac{C}{{\rm Re} \lambda (t-s)^{-1}}  }
=\displaystyle { \frac{C}{{\rm Re} \lambda} }
\end{array} \]	
admits us to define and represent the operator $U(t,s)$ by
\[
\exp \left( \int_s^t  A_1(\tau) d\tau  \right)   = \exp \left( (t-s) A_1 \right) .
\]
On the other hand, it is generally difficult to find conditions to define $\exp ( \int_s^t  A_1(\tau) d\tau  )  $ if $A_1(t)$ is $t$-dependent. 
The difficulty of exponential representation depends on whether the commutation of $U(t,s)$ and $A_1(t)$ is assumed or not.
The details are discussed in the theory of abstract evolution equations of hyperbolic type \cite{70kato,73kato}.

\section*{Appendix II: Logarithmic representation}
The validity of logarithmic representation 
\[ {\rm Log}(U(t,s)+\kappa I)=\frac1{2\pi i}\int_C{\rm Log}\,\zeta\cdot(\zeta-(U(t,s)+\kappa I))^{-1}d\zeta, \]
is confirmed by beginning with the resolvent equality
\[ \begin{array}{ll}
 (\zeta-(U(t+\Delta t,s)+\kappa I))^{-1}-(\zeta-(U(t,s)+\kappa I))^{-1} \vspace{1.5mm} \\
 =(\zeta-(U(t+\Delta t,s)+\kappa I))^{-1}\{(\zeta-(U(t,s)+\kappa I))-(\zeta-(U(t+\Delta t,s)+\kappa I))\} \vspace{1.5mm} \\
 \times(\zeta-(U(t,s)+\kappa I))^{-1}  \vspace{1.5mm} \\
 =(\zeta-(U(t+\Delta t,s)+\kappa I))^{-1}\{U(t+\Delta t,s)-U(t,s)\}(\zeta-(U(t,s)+\kappa I))^{-1},
\end{array} \]	
and therefore
\[ \begin{array}{ll}
 \frac{1}{\Delta t}\left\{(\zeta-(U(t+\Delta t,s)+\kappa I))^{-1}-(\zeta-(U(t,s)+\kappa I))^{-1}\right\}  \vspace{1.5mm} \\
 =(\zeta-(U(t+\Delta t,s)+\kappa I))^{-1}\frac{U(t+\Delta t,s)-U(t,s)}{\Delta t}(\zeta-(U(t,s)+\kappa I))^{-1}.
\end{array} \]	
It follows that
\[ \begin{array}{ll}
 \frac{1}{\Delta t}\left\{{\rm Log}(U(t+\Delta t,s)+\kappa I)-{\rm Log}(U(t,s)+\kappa I)\right\}  \vspace{1.5mm} \\
 =\frac{1}{\Delta t}\left\{\frac1{2\pi i}\int_C{\rm Log}\,\zeta\cdot(\zeta-(U(t+\Delta t,s)+\kappa I))^{-1}d\zeta-\frac1{2\pi i}\int_C{\rm Log}\,\zeta\cdot(\zeta-(U(t,s)+\kappa I))^{-1}d\zeta\right\}  \vspace{1.5mm} \\
 =\frac{1}{\Delta t}\frac1{2\pi i}\int_C{\rm Log}\,\zeta\big\{(\zeta-(U(t+\Delta t,s)+\kappa I))^{-1}-(\zeta-(U(t,s)+\kappa I))^{-1}\big\}d\zeta   \vspace{1.5mm} \\
 =\frac1{2\pi i}\int_C{\rm Log}\,\zeta\cdot(\zeta-(U(t+\Delta t,s)+\kappa I))^{-1}\frac{U(t+\Delta t,s)-U(t,s)}{\Delta t}(\zeta-(U(t,s)+\kappa I))^{-1}d\zeta
\end{array} \]
If each element in $\{A(t)\}$ commutes with each other,
\[ \begin{array}{ll}
 \frac{1}{\Delta t}\left\{{\rm Log}(U(t+\Delta t,s)+\kappa I)-{\rm Log}(U(t,s)+\kappa I)\right\}  \vspace{1.5mm} \\
 =\frac1{2\pi i}\int_C{\rm Log}\,\zeta\cdot(\zeta-(U(t+\Delta t,s)+\kappa I))^{-1}(\zeta-(U(t,s)+\kappa I))^{-1}d\zeta\frac{U(t+\Delta t,s)-U(t,s)}{\Delta t}.
\end{array} \]
Consequently, by taking a limit $\Delta t\rightarrow0$ for $x\in D(A(s))$, 
\[ \begin{array}{ll}
 \partial_t{\rm Log}(U(t,s)+\kappa I)x=\frac1{2\pi i}\int_C{\rm Log}\,\zeta\cdot(\zeta-(U(t,s)+\kappa I))^{-2}d\zeta\cdot A(t)U(t,s)x
\end{array} \]
is obtained.
Here it is readily seen that the operator $\partial_t{\rm Log}(U(t,s)+\kappa I)$ is not generally bounded on $X$.
Integrating by parts leads to
\[ \begin{array}{ll}
 \frac1{2\pi i}\int_C{\rm Log}\,\zeta\cdot(\zeta-(U(t,s)+\kappa I))^{-2}d\zeta=-\frac1{2\pi i}\int_C{\rm Log}\,\zeta\cdot\frac{\partial}{\partial\zeta}(\zeta-(U(t,s)+\kappa I))^{-1}d\zeta \vspace{1.5mm} \\
 =\frac1{2\pi i}\int_C\zeta^{-1}\cdot(\zeta-(U(t,s)+\kappa I))^{-1}d\zeta=(U(t,s)+\kappa I)^{-1},
\end{array} \]
and therefore
\[ \partial_t{\rm Log}(U(t,s)+\kappa I)x=(U(t,s)+\kappa I)^{-1}A(t)U(t,s)x. \]
By taking $U(t,s)x=y$, which is equivalent to $x=U(t,s)^{-1}y=U(s,t)y$, the equality becomes
\[ \begin{array}{ll} 
\partial_t{\rm Log}(U(t,s)+\kappa I)U(s,t)y=(U(t,s)+\kappa I)^{-1}A(t)y   \vspace{1.5mm} \\
\Leftrightarrow A(t)y=(U(t,s)+\kappa I)\partial_t{\rm Log}(U(t,s)+\kappa I)U(s,t)y.
\end{array} \]
The commutation assumption complete the derivation.
\[ \begin{array}{ll} 
 A(t)y=U(s,t)(U(t,s)+\kappa I)\partial_t{\rm Log}(U(t,s)+\kappa I)y
=(I+\kappa U(s,t))\partial_t{\rm Log}(U(t,s)+\kappa I).
\end{array} \]

\end{document}